\documentclass[12pt]{article}
\usepackage[margin=2in,includefoot,footskip=12pt,]{geometry}
\usepackage{amsfonts,amsmath, amsthm, amssymb,hyperref,fullpage}
\usepackage{graphicx}
\usepackage{listings}
\usepackage{xcolor,enumitem}
\usepackage{tikz}
\usetikzlibrary{arrows}
\usetikzlibrary{graphs,graphs.standard}
\usetikzlibrary{shapes.geometric}
\usetikzlibrary{svg.path}
\usepackage{bm}
\usetikzlibrary{patterns}

\newtheorem{theorem}{Theorem}[section]
\newtheorem{lemma}{Lemma}[section]
\newtheorem{cor}{Corollary}[section]

\newtheorem{remark}{Remark}[section]

\DeclareMathOperator{\vol}{vol}
\DeclareMathOperator{\iso}{iso}
\DeclareMathOperator{\avg}{avg}

\title{Structural and extremal properties of $l_1$-Fiedler value }

\author{
M. Rajesh Kannan \thanks{ Department of Mathematics, Indian Institute of Technology Hyderabad, Kandi, Sangareddy 502284, India. Email: rajeshkannan@math.iith.ac.in, rajeshkannan1.m@gmail.com} \and 
Rahul Roy 
\thanks{ Department of Mathematics, Indian Institute of Technology Hyderabad, Kandi, Sangareddy 502284, 502284, India. Email: ma23resch11004@iith.ac.in } 
}
\date{\today}

\begin{document}
\maketitle

\begin{abstract}
The algebraic connectivity $a(G)$, defined as the second smallest eigenvalue of the Laplacian matrix $L(G)$, admits a well-known variational characterization involving the minimization of a quadratic form subject to an $\ell_{2}$-norm constraint. In a recent work, Andrade and Dahl (2024) proposed an analogous formulation based on the $\ell_{1}$-norm, leading to the introduction of a new graph parameter $b(G)$, referred to as the $l_1$-Fiedler value. In this article, we undertake a detailed investigation of the structural and extremal properties of $b(G)$. We first derive a Nordhaus--Gaddum type inequality for $b(G)$. For trees, we determine both global maximizer and minimizers of $b(G)$, and present extremal constructions for trees with prescribed diameter, maximum degree, and number of pendant vertices. We further establish a connection between $b(G)$ and Laplacian matrices, and obtain a bound for $b(G)$ in terms of the edge connectivity, along with a complete characterization of the graphs attaining equality. We derive an explicit formula that describes the behaviour of $b(G)$ under the addition of pendant vertices. We also investigate the connection between $b(G)$ and the isoperimetric number.

\end{abstract}
\section{Introduction}

The eigenvalues and eigenvectors of Laplacian matrices of graphs play a fundamental role in spectral graph theory and have found wide-ranging applications in areas such as combinatorial optimization, electrical networks, graph signal processing, data science, and machine learning. In particular, the second smallest eigenvalue of the Laplacian matrix serves as a robust measure of the connectivity of a graph \cite{alg-con}. It also provides useful lower or  upper bounds for several important graph invariants, including the isoperimetric number, maximum cut, independence number, genus, and diameter \cite{mohar-lap-sur-1, mohar-eigen-comb-opti, trevisan_spec-partition}. The associated eigenvector plays a central role in graph partitioning and data clustering, most notably in spectral clustering methods \cite{trevisan_spec-partition, ulrike-spectral}. In his seminal work \cite{alg-con}, Fiedler introduced the term algebraic connectivity to describe the second smallest Laplacian eigenvalue.

The algebraic connectivity of a graph $G$, denoted by $a(G)$,  admits a well-known variational characterization involving the $l_2$-norm in the objective function and the constraints. More recently, in \cite{enide-geir-2024}, Andrade and Dahl investigated analogous optimization formulations based on the $l_1$-norm and the $l_\infty$-norm. The resulting parameters are denoted by $b(G)$ and $\gamma (G)$, respectively. They established a formula for $b(G)$  in terms of the sparsest cut of $G$, which implies that computing $b(G)$ is an NP-hard problem. In contrast, they showed that $\gamma (G)$ can be computed via a linear programming formulation. 

 Let \(G = (V(G), E(G))\) be an unweighted, undirected, and simple graph. The set \(V(G)\) contains the vertices, and \(E(G)\) contains the edges of \(G\). The number of vertices, written as \(|V(G)|\), is \(n\). If two vertices \(u\) and \(v\) are adjacent, we write \(u \sim v\).  The \textit{adjacency matrix} of \(G\), denoted by \(A(G)\), is the symmetric \(n \times n\) matrix whose \((u,v)\)-entry is defined as
\[
a_{uv} = 
\begin{cases}
1, & \text{if } u \sim v,\\[4pt]
0, & \text{otherwise}.
\end{cases}
\]
 For a graph $G$, let $\Delta(G)$ denote the \textit{diagonal matrix} with the degrees of the vertices of $G$ as its diagonal elements.  The \textit{Laplacian matrix} $L(G)$ is the $n\times n$  matrix defined as 
$$
L(G)=\Delta(G)-A(G),
$$
The eigenvalues of the adjacency matrix help in understanding several structural properties of a graph, such as whether it is bipartite or regular. They also give bounds for important graph parameters, including the chromatic number, clique number, and independence number \cite{ipm-brouwer-haemers}.
The eigenvalues of the Laplacian matrix of a graph reveal information about the connectedness of the graph $G$.  It is well known that $G$ is connected if and only if $a(G) > 0$ \cite{alg-con, fiedler-1989}.
An eigenvector corresponding to $a(G)$ is called a \textit{Fiedler vector} \cite{enide-geir-2024}. The value \(a(G)\) can be computed using a standard optimization formula based on the Courant–Fischer theorem \cite{matrix-analysis}. This formula is:
 \begin{equation*}
a(G)=\min  \biggl\{\sum_{uv \in E(G)}(x_u - x_v)^2: \sum_{v \in V(G)} x_v =0 ~\mbox{and}~\Vert x \Vert _{2}=1\ \biggl\},
 \end{equation*} 
 where $\Vert x \Vert_2 := \sqrt{\sum\limits_{i=1}^n \vert x_i \vert^2}$ is the $l_2$-norm of the vector $x\in \mathbb{R}^n$. 
The vector \(x\) that gives the minimum in this expression has an important property: the squared differences \((x_u - x_v)^2\) along the edges are as small as possible. This means that the values of \(x\) change in an optimal and smooth way across the edges of the graph. Andrade and Dahl \cite{enide-geir-2024} referred to this as the \( \ell_2 \)-smoothing problem. They also studied similar optimization problems using the \( \ell_1 \)-norm and the \( \ell_\infty \)-norm \cite{enide-geir-2024}. This led to the definition of two new graph parameters:

$$
b(G)=\min\biggl\{\sum_{uv \in E(G)}\vert x_u - x_v\vert: \sum_{v \in V(G)}x_v=0~\mbox{and}~ \Vert x \Vert_{1}=1\biggl\},
$$
$$
\gamma(G)=\min\biggl\{\max_{uv \in E(G)}\vert x_u - x_v\vert: \sum_{v \in V(G)} x_v =0 ~\mbox{and}~\Vert x \Vert _{\infty}=1\biggl\}.
$$
Here, $\Vert x \Vert_1 := \sum\limits_{i=1}^n |x_i|$ denotes the $l_1$-norm, and $\Vert x \Vert_\infty := \max \{ |x_i| : 1 \leq i \leq n \}$ denotes the $l_\infty$-norm of the vector $x \in \mathbb{R}^n$. These optimization problems are well-defined because the constraint sets are compact and the objective functions are continuous. We refer to $b(G)$ as the $l_1$-Fiedler value. The vectors that achieve the minima are called the \(\ell_1\)-Fiedler vector and the \(\ell_\infty\)-Fiedler vector, respectively. 
 

The main objective of this article is to further investigate the quantity $b(G)$. We first establish a bound on $b(G)$ in terms of the number of vertices and edges of a graph (Theorem~\ref{edge-bound-b(g)}). Using this result, we derive a Nordhaus--Gaddum type bound for $b(G)$ (Theorem~\ref{nord-gaddum}).  For trees, we identify the global maximizer and minimizers for $b(G)$ (Theorem~\ref{extremal trees}), and construct extremal examples for trees with prescribed diameter (Theorem~\ref{diameter-extremizer}), maximum degree (Theorem~\ref{extremal-graph-max-vertex-degree}), and number of pendant vertices (Theorem~\ref{pendant-vertex-extremizer}).  A connection between $b(G)$ and the Laplacian matrices of graphs is established in Theorem~\ref{lap-conn}.  We then obtain a bound for $b(G)$ in terms of the edge connectivity of $ G $ (Theorem~\ref{edge-connc-b(g)}), and provide a complete characterization of the graphs for which this bound is sharp (Theorem~\ref{edge-connc-b(g)-equi-chara}). Next, we derive a formula describing the change in $b(G)$ under the addition of pendent vertices (Theorem~\ref{vertex-add}). In Section~\ref{iso-connection}, we investigate the relationship between the isoperimetric number and the quantity $b(G)$.

\section{Preliminary results}

We denote by \(K_n, C_n, P_n,\) and \(S_n\) the complete graph, cycle graph, path graph, and star graph on \(n\) vertices, respectively. The degree of a vertex $v$ in a graph is denoted by $\deg(v)$. A vertex $v$ of a graph $G$ is called a pendant vertex if $\deg(v)= 1$.  Let  $d_{\min}(G)$ denote the minimum vertex degree of $G$, and $d_{\max}(G)$ is the maximum degree of any vertex in $G$.  A graph $G$ is said to be \textit{regular} if all its vertices have the same degree. That is,  $d_{\max}(G)=d_{\min}(G)$. The average degree of a graph $G$, denoted by $d_{\avg}(G)$, is the average of the degrees of all the vertices in $G$. It is given by:
$$
d_{\avg}(G)=\frac{1}{\vert V(G) \vert} \sum\limits_{u \in V(G)} \deg(u).
$$
If $G$ has $n$ vertices and $m$ edges, then $d_{\avg}(G)=\frac{2m}{n}.$

  For a non-empty subset $S$ of $V(G)$ with $\vert S \vert < n$, the \textit{cut} induced by $S$, denoted by $\partial S_G$, is
defined as follows: $$\partial S_G = \biggl\{uv \in E(G) : u \in S ~\mbox{and}~v \in S^c\biggl\},$$ where $S^c = V(G) \setminus S$. 
We write $\partial S$ in place of $\partial S_G$ when the underlying graph is clear from the context.

The \textit{relative cut-size} is defined as $$\xi (S) =  \frac{\vert \partial S \vert }{\vert S \vert}.$$  The \textit{isoperimetric number} of $G$ is defined as: $$\iso(G) = \min_S \xi (S),$$ where the minimum is taken over all subset $S$ of $V(G)$ with $0 < \vert S \vert  \leq   \lfloor \frac{n}{2}\rfloor.$  A set $S \subseteq V(G)$ with $\vert S \vert \leq  \lfloor \frac{n}{2}\rfloor $ is \textit{isoperimetric} if $\iso(G) = \xi (S).$\\
The \textit{edge density} or \textit{sparsity} of a cut induced by $S$, denoted by $\rho(S)$, is defined as $$\rho(S) = \frac{\vert \partial(S) \vert }{ \vert S \vert \vert S^c \vert}.$$ 
A cut is a \textit{sparsest cut} if it minimizes $\rho(S)$. 
\begin{theorem}[{\cite{enide-geir-2024}}]\label{main theorem} Let $G$ be a graph on $n$ vertices. Then, 
     $$ b(G)=\frac{n}{2} \min_S \rho(S),$$
where the minimum is taken over the non-empty subsets $S$ of $V(G)$ such that $S \neq V(G)$ and both $S$ and $S^c$ induce connected subgraphs of $G$. 
\end{theorem}

Next, we collect some of the known bounds on $b(G)$ in terms of other graph parameters. The proofs can be found in \cite{enide-geir-2024}.

\begin{theorem}[{\cite{enide-geir-2024}}]\label{b(g)-min-deg-cutsize}  Let $G$ be a graph with $m$ edges, and $\lambda_1(G)$ denote the largest eigenvalue of $L(G)$. Then we have the following: 
\begin{itemize}
\item[(a)] $\frac{a(G)}{2} \leq b(G) \leq \frac{\lambda_1(G)}{2}$.
\item[(b)] $\min\limits_S \xi(S) \leq b(G) \leq \frac{n}{2(n-1)}d_{\min}(G).$
\item[(c)] $b(G)\leq \sqrt{ma(G)}.$

\end{itemize}    
\end{theorem}

 We can find the sparsest cut for trees explicitly as follows \cite{enide-geir-2024}: Let $T=(V, E)$ be a tree, and let $u v \in E $. Then $T \backslash\{uv\}$ consists of two disjoint trees $T_u$ and $T_v$ defined as follows: $T_u=(V_u, E_u)$ is the subtree containing the vertex  $u$, and  $T_v=(V_v, E_v)$ is the subtree containing the vertex $v$. An edge $uv$ is a center edge if $ \Bigl|\vert V_u\vert-\vert V_v\vert \Bigl|\,$is the smallest possible.
 

 \begin{theorem}[{\cite[Theorem 5.1.]{enide-geir-2024}}]\label{tree case}
     Let $T=(V, E)$ be a tree. Then there exists a centre edge $uv$ such that the cut induced by $V_u$ is a sparsest cut, and 
\begin{equation*}
b(T)=\frac{1}{2}\left(\frac{1}{\vert V_u \vert} +\frac{1}{\vert V_v\vert} \right).
\end{equation*} 
 \end{theorem}

A \textit{substar} in a tree $T$ is a vertex-induced subgraph of $T$ that is a star graph.

\begin{theorem}[{\cite[Corollary 5.2]{enide-geir-2024}}]\label{substar}
Let $T=(V, E)$ be a tree. Then the set of centre edges forms a substar.    
\end{theorem}
Next, we recall some well-known graphs for which $b(G)$ is calculated in \cite{enide-geir-2024}. 

\begin{theorem}[{\cite{enide-geir-2024}}]\label{b_cycle_path_star_compl}
    \begin{enumerate}
        \item $b(K_n)=\frac{n}{2}$
        \item $b(C_n)=\frac{4}{n}$ if $n\geq 4$ and $n$ is even, and $b(C_n)=\frac{n}{\lfloor \frac{n}{2}  \rfloor. \lceil \frac{n}{2} \rceil }$ if $n \geq 3$ and $n$ is odd.
        \item $b(P_n)=\frac{2}{n}$ if n is even and $b(P_n)=\frac{2n}{n^2 -1}$ if n is odd.
        \item $b(S_n)= \frac{1}{2}+\frac{1}{2(n-1)}$.
    \end{enumerate}
\end{theorem}

The complement of a graph $G$, denoted by $G^c$, is the graph with the vertex set $V(G)$ and two distinct vertices of $G^c$ are adjacent if and only if they are not adjacent in $G$.
The join of $G$ and $H$, denoted by $G \vee H$, is the graph with vertex set $V(G) \cup V(H)$ and the edge set is given by
$$
E(G \vee H)= E(G) \cup E(H) \cup  \{uv| u \in V(G), v \in V(H)\}.
$$
\begin{theorem}[{\cite{MR4281912}\label{alg-low-bound}}]
    Let $G$ be a graph with $V(G)=n \geq 2$. Then
    $$
a(G)+a(G^c) \geq 1
    $$
    The equality holds if and only if $G$ or $G^c$ is isomorphic to the join of an isolated vertex and a disconnected graph of order $n-1$.
\end{theorem}


\section{Nordhaus–Gaddum type bound}

To start with, we derive an upper bound for $b(G)$ in terms of the number of vertices and edges of the graph $G$, and characterize the extremal graphs. 
\begin{theorem}\label{edge-bound-b(g)}
    Let $G$ be a graph on $n$ vertices and $m$ edges. Then  $$b(G)\leq\frac{m}{n-1}.$$ Furthermore, equality holds if and only if $ G \cong K_{n}$.
    \label{3}
\end{theorem} 

\begin{proof}
By Theorem   \ref{b(g)-min-deg-cutsize}, we have    

    \begin{align*}
   b(G) &\leq \frac{n}{2(n-1)}d_{\min}(G)\\ 
   &\leq \frac{n}{2(n-1)}d_{\avg}(G) \\
   &=\frac{n}{2(n-1)}  \frac{2m}{n}\\
   &=\frac{m}{n-1}.
     \end{align*}
        
  If equality holds in the above, then $d_{\min}=d_{\avg}$. That is, the graph $G$ is regular. 
    
   Let $G$ be an $r$-regular graph.  Then $$b(G)=\frac{nr}{2(n-1)},$$ and, by Theorem \ref{main theorem}, the minimum edge density is  $\frac{r}{n-1}$. Let $u$ and $v$ be two adjacent vertices, and $H$ be the subset consisting only of the vertices $u$ and $v$, along with the edge connecting them. Then
     $$ \rho(H)=\frac{2r-2}{2(n-2)}=\frac{r-1}{n-2}. $$ Since
    $$\frac {r-1}{n-2} \geq \frac{r}{n-1},$$ we obtain $$nr-n-r+1 \geq nr-2r.$$ That is, $ r\geq n-1$. Thus, $r=n-1$, and hence $G\cong K_n$.
    
    If $G \cong K_n$, then, by Theorem \ref{b_cycle_path_star_compl}, we have 
    $$b(K_n)=\frac{n}{2}=\frac{n(n-1)}{2(n-1)}=\frac{m}{n-1}.$$
    
\end{proof}

    From Theorem \ref{edge-bound-b(g)}, we can conclude that among all the connected graphs on $n$ vertices, the complete graph $K_n$ is the unique maximizer for $b(G)$.  Next, we discuss the minimizers.
    
    \begin{theorem}\label{lower bound}
    Let $G$ be a graph on $n$ vertices. Then $b(G) \geq b(P_n)$, where $P_n$ denotes the path graph on $n$ vertices.
    \end{theorem} 
    
    \begin{proof}
    Let $S$ be a subset of $V(G)$ that induces a sparsest cut. Then $$b(G)=\frac{n}{2} \frac{\vert \partial S \vert}{\vert S \vert \vert S^c\vert}.$$ The minimum value of $b(G)$ is obtained when the numerator is as small as possible and the denominator is as large as possible. Note that, the minimum value for $\vert \partial S \vert$  is $1$,  and the maximum value of $\vert S \vert \vert S^c \vert $ is $\frac{n^2}{4}$ if $n$ is even, and $\frac{n^2-1}{4}$ if $n$ is odd.  These two conditions are simultaneously satisfied for the path graph $P_n$.
   \end{proof}

      \begin{remark}
        Note that the path graph $P_n$ is not a unique minimiser for $b(G)$.  Consider a pendant vertex of $P_n$, $n \geq 6$. If this vertex is removed and attached to another vertex on the same side of the centre edge, then $b(G)$ remains unchanged. In fact, later in Section \ref{extremal section}, we will see the construction of other minimisers for $b(G)$. 
        \end{remark}

Note that for any graph $G$ on $n$ vertices,  Theorem \ref{edge-bound-b(g)} implies that $b(G) \leq \frac{n}{2}$. The following result states that even the sum of $b(G)$ and $b(G^c)$ is at most $\frac{n}{2}$. This is a Nordhaus-Gaddum-type bound for $b(G)$.

\begin{theorem}\label{nord-gaddum}
    Let $G$ be a connected graph on $n$ vertices. Then $$\frac{1}{2} < b(G)+ b(G^c) \leq \frac{n}{2}.$$ The equality on the right side holds if and only if $G \cong K_{n}$.

\end{theorem}

\begin{proof}
By Theorem \ref{edge-bound-b(g)}, we have $b(G)\leq \frac{m}{n-1}$ and $b(G^c)\leq \frac{1}{n-1}(\frac{n(n-1)}{2}-m)$. Adding both inequalities gives the desired result. Equality holds if and only if $b(G)= \frac{m}{n-1}$ and $b(G^c)= \frac{1}{n-1}(\frac{n(n-1)}{2}-m)$. By Theorem \ref{edge-bound-b(g)}, $G$ must be $K_n$.

From Theorem \ref{b(g)-min-deg-cutsize} and Theorem \ref{alg-low-bound},  we obtain
$$
b(G)+ b(G^c) \geq \frac{a(G)+ a(G^c)}{2} \geq \frac{1}{2}.
$$
 By Theorem \ref{alg-low-bound}, the equality on the right side holds if and only if $G$ or $G^c$ is isomorphic to the join of an isolated vertex and a disconnected graph on $(n-1)$ vertices. 
Without loss of generality, let $G=\{v\} \vee H$ where $H$ is a disconnected graph on $(n-1)$ vertices. Let $S \subseteq V(G)$ induces a sparsest cut in $G$ with $\vert S \vert \leq \frac{n}{2}$. 
If $S$ does not include the vertex $v$, then  $|\partial S| \geq |S| $ as every vertex of $S$ is adjacent to $v$. If $S$ includes $v$, then also $\vert \partial S\vert \geq n-\vert S \vert\geq\vert S \vert$. Therefore
$$
b(G)=\frac{n}{2} \frac{|\partial S|}{|S||S^c|} \geq \frac{n}{2} \times \frac{1}{|S^c|} >\frac{1}{2}.$$

Thus, $$ b(G)+ b(G^c) > \frac{1}{2}.$$
\end{proof}

\begin{remark}
    Note that the lower bound in the previous theorem is asymptotically tight, i.e., there exists a class of graphs for which the bound is achieved asymptotically as $n \rightarrow \infty$. For, consider the star graph on $n$ vertices. Then,
    $$
b(S_n)+b(S_n ^c)= b(S_n) = \frac{1}{2}+\frac{1}{2(n-1)},
    $$
    as $S_n^c$ is disconnected. Hence $b(S_n)+b(S_n ^c) \rightarrow \frac{1}{2}$ as $n \rightarrow \infty$.
    \end{remark}

\section{Extremal Trees}\label{extremal section}
In this section, we study extremal problems for trees. We begin by proving a strengthened version of Theorem \ref{tree case}. The proof is similar to that of Theorem \ref{tree case} in \cite{enide-geir-2024}, for the sake of completeness we include a proof here. Let $T$ be a tree, and $uv \in E$. Define $T_u=(V_u, E_u)$ and $T_v=(V_v, E_v)$ are the two disjoint subtrees obtained by removing the edge $uv$ such that $u \in V_u$ and $v \in V_v$. 
\begin{theorem}\label{extended centre edge result}
    Let $T=(V, E)$ be a tree. Let $uv \in E$. Then 
    $$
b(T)=\frac{1}{2}\left(\frac{1}{\vert V_u \vert} +\frac{1}{\vert V_v\vert} \right)
    $$
    if and only if $uv$ is a centre edge in $T$.
\end{theorem}

\begin{proof}
   Let $uv \in E$ and 
   $$
b(T)=\frac{1}{2}\left(\frac{1}{\vert V_u \vert} +\frac{1}{\vert V_v\vert} \right).
   $$
  Suppose that $uv$ is not a centre edge of $T$. Then, there exists an edge $u^\prime v^\prime$ in $T$ such that $ \Bigl|\,\vert V_{u^\prime} \vert-\vert V_{v^\prime}\vert \Bigl|\, < \Bigl|\,\vert V_u\vert-\vert V_v\vert \Bigl|\,$. Since $\vert V_{u^\prime}\vert+\vert V_{v^\prime}\vert=\vert V_{u}\vert+\vert V_{v}\vert=n$, we must have $\vert V_{u^\prime} \vert \vert V_{v^\prime}\vert >\vert V_{u}\vert \vert V_{v}\vert$. Therefore 
   $$
   \frac{n}{2} \rho(V_{u^\prime})=\frac{n}{2}\frac{1}{\vert V_{u^\prime} \vert \vert V_{v^\prime}\vert}< \frac{n}{2}\frac{1}{\vert V_{u}\vert \vert V_{v}\vert}=b(T),
   $$
   which is a contradiction.
   
   Conversely, let $uv$ be a centre edge in $T$. Then, by Theorem \ref{main theorem},  $b(T) \leq \frac{1}{2}\left(\frac{1}{\vert V_u \vert} +\frac{1}{\vert V_v\vert} \right).$ Suppose that $b(T) <\frac{1}{2}\left(\frac{1}{\vert V_u \vert} +\frac{1}{\vert V_v\vert} \right) .$ 
   Let $S \subseteq V$ be such that both the subgraphs induced by $S$ and $S^c$ are connected  in $T$,  and $b(T)=\frac{n}{2} \frac{\vert\partial S \vert}{\vert S \vert \vert S^c \vert}.$ Since connected subgraphs of a tree are also a tree, both the induced subgraphs are subtrees of $T$. Hence $\vert \partial S\vert=1$. Therefore, $S$ and $S^c$ are obtained by deleting an edge, say  $wz$, from $T$. Without loss of generality, let $S=V_w$ and $S^c=V_z.$ Then,

   \begin{align*}
 \frac{n}{2} \frac{1}{\vert V_{u}\vert \vert V_{v}\vert} &= \frac{1}{2}\left(\frac{1}{\vert V_u \vert} +\frac{1}{\vert V_v\vert} \right)\\
    &>b(T)\\
    &=\frac{n}{2} \frac{\vert\partial S \vert}{\vert S \vert \vert S^c \vert}\\
    &=\frac{n}{2} \frac{\vert \partial V_w \vert}{\vert V_w \vert \vert V_z \vert }\\
    &=\frac{n}{2} \frac{1}{\vert V_w \vert \vert V_z \vert}.
\end{align*}
   Therefore $\vert V_{u}\vert \vert V_{v}\vert < \vert V_w \vert \vert V_z \vert$. Again, by using $\vert V_w\vert+\vert V_z\vert=\vert V_{u}\vert+\vert V_{v}\vert=n$, it follows that 
    $ \Bigl|\,\vert V_w \vert-\vert V_z\vert \Bigl|\, < \Bigl|\,\vert V_u\vert-\vert V_v\vert \Bigl|\,$, contradicts the fact that $uv$ is a centre edge in $T$.
\end{proof}

Next, we study the global extremals for the quantity $b(G)$ among all the trees.
\begin{theorem}\label{extremal trees}
Among all trees of $n$ vertices, the path graph $P_n$ minimizes $b(T)$ and the star graph $S_n$ maximizes $b(T)$. Moreover, $S_n$ is the unique tree that attains the maximum.
\end{theorem}

\begin{proof}
The first part of the theorem is follows from \ref{lower bound}.
Let $T$ be a tree on $n$ vertices.  From Theorem \ref{extended centre edge result}, we have $b(T)=\frac{n}{2} \frac{1}{\vert V_{u}\vert \vert V_{v}\vert}$, where $uv$ is a centre edge in $T$. Since $\vert V_{u}\vert \vert V_{v}\vert \geq n-1$, we get $$b(T) \leq \frac{n}{2(n-1)}=b(S_n),$$
by using Theorem \ref{b_cycle_path_star_compl}. Hence, $S_n$ maximizes $b(T)$.

Let $T=(V,E)$ be a tree such that $b(T)=b(S_n)$. Let $uv$ be a centre edge of $T$. Then, by Theorem \ref{extended centre edge result}, $\vert V_{u}\vert \vert V_{v}\vert=n-1$. Without loss of generality, let $\vert V_u\vert=1$ and $\vert V_v \vert=n-1$. Then $ \vert V_v\vert-\vert V_u\vert=n-2$. Then $\Bigl| \vert V_{u^\prime} \vert - \vert V_{v^\prime}\vert \Bigl| \geq n-2$ for every other edge $u^\prime v^\prime$ in $T$. Since the quantity $\Bigl| \vert V_{u^\prime}\vert -\vert V_{v^\prime}\vert \Bigl| $ is at most $n-2$ for any edge $u^\prime v^\prime$ in T, we should have $\Bigl| \vert V_{u^\prime}\vert -\vert V_{v^\prime}\vert \Bigl| =n-2$ for all edges in $T$. Thus every edge of $T$ is a centre edge, and hence, by Theorem \ref{substar},  $T$ must be $S_n$.
\end{proof}

\begin{remark}
\label{1}
Let $T=(V, E)$ be a tree on $n$ vertices. Let $uv$ be a centre edge of $T$. Note that,  $ \Bigl| \vert V_u\vert-\vert V_v\vert \Bigl|=n-2$ if and only if $T \cong S_n$. So, $S_n$ is the only tree where sparsest cut is induced by a singleton set.

\end{remark}
\begin{remark}
\label{2}
    Among all trees on $n$ vertices, the path graph $P_n$ is not the unique tree that minimizes $b(T)$. For example, the following tree on six vertices also attains the minimum value of $b(T)$.
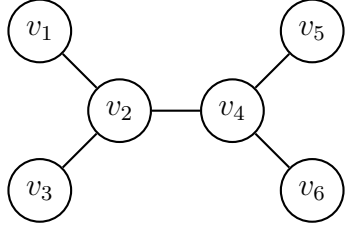
$$    
    \begin{tikzpicture}[node distance={15mm}, thick, main/.style = {draw, circle}] 
\node[main] (1) {$v_1$};
\node[main] (2) [ below right of=1] {$v_2$};
\node[main] (3) [below left of=2] {$v_3$}; 
\node[main] (4) [ right of=2] {$v_4$};
\node[main] (5) [above right of=4] {$v_5$};
\node[main] (6) [below right of=4] {$v_6$};

\draw (2) -- (4);
\draw (2) -- (1);
\draw (2) -- (3);
\draw (5) -- (4);
\draw (6) -- (4);

\end{tikzpicture}
$$
\end{remark}
We now prove a couple of results that facilitate the analysis of extremal trees with fixed diameter, maximum degree, and number of pendant vertices.
\begin{lemma}\label{centre-edge-domination}
    Let $T=(V, E)$ be a tree. Let $wz\in E$, and let the subtree $T_w$ contain a centre edge of $T$. Then $\vert V_w \vert  > \vert V_z \vert$. 
\end{lemma}
\begin{proof}
  Let $uv$ be a centre edge of $T$, which is in $T_w$. As $T$ is a tree, there exists a unique path between the vertices $u$ and $w$, say $P$. If $v$ lies in the path $P$, then $d(w,v)=d(w,u)-1$. Otherwise, $P \cup uv$ is a path from $w$ to $v$. Since it is a unique path from $w$ to $v$, we have   $d(w,v)=d(w,u)+1$. So, $d(w,v) \neq d(w,u)$.
    Without loss of generality, let $d(w,v) < d(w,u)$. Then $\vert V_w \vert > \vert V_u \vert$ and $\vert V_v \vert > \vert V_z \vert.$ Note that, $V_w$ and $V_z$ are the components obtained from $T$ by deleting the edge $wz$, and $V_u$ and $V_v$ are the components obtained from $T$ by deleting the edge $uv$.
    
    Suppose that $\vert V_w \vert  \leq \vert V_z \vert$. Then $\Bigl| \vert V_w \vert - \vert V_z\vert \Bigl|=\vert V_z \vert - \vert V_w \vert.$ Since $uv$ is a centre edge in $T$, we have 
    $$
\vert V_z \vert - \vert V_w \vert \geq \Bigl| \vert V_u \vert - \vert V_v\vert \Bigl|\geq \vert V_v \vert - \vert V_u\vert.
    $$
Therefore $\vert V_z \vert \geq \vert V_v \vert + \vert V_w \vert - \vert V_u\vert>\vert V_v \vert$, a contradiction. Thus $\vert V_w \vert  > \vert V_z \vert$. 
    \end{proof}
    \begin{lemma}\label{centre edge unchanged}
 Let $uv$ be a centre edge in the path graph $P_n = (V, E)$. Let $T^\prime =(V^\prime,E^\prime)$ be a tree obtained by attaching  $r$ pendant vertices to either $u$ or $v$. Then $uv$ is a centre edge in $T^\prime$. 
    \end{lemma}
    \begin{proof}
Suppose we add $r_1$ vertices to $u$ and $r_2$ vertices to $v$.
Then 
$$
\Bigl| \vert V^\prime_u \vert - \vert V^\prime_v\vert \Bigl|=\Bigl| \vert V_u \vert+r_1 - \vert V_v\vert-r_2 \Bigl|\leq \vert r_1-r_2\vert+\Bigl| \vert V_u \vert - \vert V_v\vert \Bigl| \leq \vert r_1-r_2\vert+1.
$$
Let $n \geq 3$.
First, consider the edges that are newly added to $P_n$. Let $ux$ be such an edge with $\deg(x)=1$ in $T^\prime.$ Then 
$$
\vert V^\prime _u \vert -\vert V^\prime_x\vert=n+r-2\geq r_1+r_2+1 \geq \vert r_1-r_2\vert+1.
$$
Let $wz \in E^\prime \setminus \{uv\}$ be an edge lying on the path $P_n$. Then exactly one of the components $T_w$ or $T_z$ obtained by deleting $wz$ contain the edge $uv$. Without loss of generality, assume that $T_w$ contain the edge $uv$. By Lemma \ref{centre-edge-domination}, it follows that $\vert V_w\vert > \vert V_z \vert$.
Therefore
$$
\Bigl| \vert V^\prime_w \vert - \vert V^\prime_z\vert \Bigl| \geq \vert V^\prime_w \vert - \vert V^\prime_z\vert= \vert V_w \vert +r -\vert V_z \vert\geq r_1+r_2+1 \geq \vert r_1-r_2\vert+1.
$$
If $n=2$, then $P_2=uv$. In this case, every edge of $T^{'}$ except $uv$ is incident with either $u$ or $v$. Let $ux^\prime$ be such an edge, where $\deg(x^\prime)=1$. Then
$$
\vert V^\prime _u \vert -\vert V^\prime_{x ^\prime}\vert=r=r_1+r_2 \geq \vert r_1-r_2\vert=\Bigl| \vert V_u \vert - \vert V_v\vert \Bigl|.
$$
Hence,  $uv$ is a centre edge in $T^\prime$.

    \end{proof}
    \begin{lemma}\label{pendant-vertex-smaller-set-domination}
    Let $T=(V, E)$ be a tree, and $uv$ be a centre edge such that $\vert V_u \vert \leq \frac{n}{2}$. Let $T^\prime=(V^\prime, E^\prime)$ be a tree obtained by adding a pendant vertex to $T$. Let $u^\prime v^\prime$ be a centre edge of $T^\prime$ and $\vert V^\prime_{u^\prime} \vert\leq \frac{n+1}{2}.$ Then $\vert V_u \vert \leq \vert V^\prime_{u^\prime}\vert.$
\end{lemma}
\begin{proof}
    Since $u^\prime v^\prime$ is a centre edge in $T^\prime$ and $\vert V^\prime_{u^\prime}\vert \leq \vert V^\prime_{v^\prime}\vert$ , we have $  \vert V^\prime_{v^\prime}\vert - \vert V^\prime_{u^\prime}\vert \leq \Bigl| \vert V^\prime_{v}\vert - \vert V^\prime_{u}\vert \Bigl|$. Now note that exactly one of the following cases holds:
    $$
\vert V^\prime_v \vert=\vert V_v \vert+1 \text{ and } \vert V^\prime_u \vert=\vert V_u \vert,
    $$
    or
    $$
\vert V^\prime_v \vert=\vert V_v \vert \text{ and } \vert V^\prime_u \vert=\vert V_u \vert+1.
    $$
    Hence $$\Bigl| \vert V^\prime_{v}\vert - \vert V^\prime_{u}\vert \Bigl| \leq \Bigl| \vert V_{v}\vert - \vert V_{u}\vert \Bigl|+1=\vert V_{v}\vert - \vert V_{u}\vert+1.$$
    Also $\vert V^\prime_{v^\prime}\vert - \vert V^\prime_{u^\prime}\vert=n+1-2\vert V^\prime_{u^\prime}\vert$ and $\vert V_{v}\vert - \vert V_{u}\vert =n-2\vert V_{u}\vert$. Therefore,
    $$
n+1-2\vert V^\prime_{u^\prime}\vert\leq n-2\vert V_{u}\vert+1,
    $$
   That is,   $\vert V_u \vert \leq \vert V^\prime_{u^\prime}\vert.$

\end{proof}
In the following theorem, we study extremizers for trees with a fixed diameter.
\begin{theorem}\label{diameter-extremizer}
    Let $T$ be a tree on $n$ vertices with fixed diameter $D(\geq 3)$. Then 
    $$
\frac{n}{2\lceil \frac{n}{2}\rceil \lfloor \frac{n}{2}\rfloor} \leq b(T) \leq \frac{n}{2\lceil \frac{D}{2} \rceil(n-\lceil \frac{D}{2} \rceil)}
    .$$ Moreover, the both the bounds are sharp.

\end{theorem}
\begin{proof}
  Construct a tree $T$ which minimizes $b(T)$ as follows: Start with the path on $D+1$ vertices. If $D$ is odd, the has a unique centre edge $uv$. We attach pendant vertices alternately to 
 $u$ and $v$  until all vertices are exhausted. If $D$ is even, without loss of generality, assume that $\vert V_u \vert=\vert V_v \vert-1$. We first attach a pendant vertex to $u$, and then alternately to $v$ and $u$, until all vertices are exhausted. In both cases, by Lemma \ref{centre edge unchanged}, $uv$ still remains the centre edge. Moreover, when we remove the edge $uv$, we obtain two trees with almost equal sizes. 
   Thus, $b(T)=\frac{n}{2\lceil \frac{n}{2}\rceil \lfloor \frac{n}{2}\rfloor}$ . Therefore, the tree $T$ minimizes $b(T)$ among all trees in this case with diameter $d$.



Next, we construct a tree that maximizes $b(T)$. Begin with the path on ${D+1}$ vertices. Now, if $D$ is even, then there is a unique middle vertex of the path; we attach all remaining vertices as pendant vertices to this vertex. If $D$ is odd, the path has a unique centre edge, and we attach all remaining vertices as pendant vertices to exactly one endpoint of this edge. By Lemma \ref{centre edge unchanged}, the centre edge of $P_{D+1}$ remains a central edge throughout the construction. Consequently, in both cases, a sparsest cut is induced by $\lfloor \frac{D+1}{2} \rfloor$ vertices, i.e.,$\lceil \frac{D}{2} \rceil$ vertices. Therefore, we have $b(T)=\frac{n}{2\lceil \frac{D}{2} \rceil (n-\lceil \frac{D}{2} \rceil)}$.

  Note that any tree $T^\prime=(V^\prime, E^\prime)$ of diameter $D$ can be constructed by successively attaching pendant vertices to a path on $D+1$ vertices.  By Lemma \ref{pendant-vertex-smaller-set-domination}, whenever a pendant vertex is attached to a tree, the removal of a centre edge in the resulting tree produces a smaller component whose order is at least that of the smaller component obtained by removing a centre edge in the original tree. Therefore, at each step of the process, the resulting tree admits a sparsest cut induced by at least $\lceil \frac{D}{2} \rceil$ vertices, and this cut contains at most half of the vertices of the resulting tree. At the end of this process, we have a sparsest cut induced by at least $\lceil \frac{D}{2} \rceil$ vertices, and containing  at most $n/2$ vertices in $T^\prime$.
  
  Let $u^\prime v^\prime$ be a centre edge in $T^\prime$, and let $V^\prime_{u^\prime}$ induce a sparsest cut in $T^\prime $, with $\lceil \frac{D}{2} \rceil \leq \vert V^\prime_{u^\prime} \vert \leq \frac{n}{2}.$ Then $\vert V^\prime_{u^\prime} \vert \vert V^\prime_{v^\prime}\vert \geq \lceil \frac{D}{2} \rceil (n-\lceil \frac{D}{2} \rceil) $. Hence 
  $$
b(T^\prime)=\frac{n}{2}\frac{1}{\vert V^\prime_{u^\prime} \vert \vert V^\prime_{v^\prime}\vert}\leq \frac{n}{2\lceil \frac{D}{2} \rceil (n-\lceil \frac{D}{2} \rceil}.
  $$
\end{proof}

Let $T=(V,E)$ be a tree with at least two centre edges. By Theorem \ref{substar}, we know that centre edges form a star subgraph. We define the  star-root vertex of $T$ to be the centre vertex of this star. If $T$ has a unique centre edge, say $uv$, we call $u$(or $v$) a star-root vertex if the subtree $T_u$(or $T_v$) has size at least as large as that of $T_v$(or $T_u$). In particular,, if $\vert V_u \vert =\vert V_v \vert$,  both $u$ and $v$ are called the star-root vertices of $T$. The following lemma will be useful.

\begin{lemma}\label{centre vertex domination}
    Let $T=(V,E)$ be a tree and $u$ be a star-root vertex of  $T$. If $uv \in E$, then $\vert V_u \vert \geq \vert V_v\vert.$
\end{lemma}
\begin{proof}
    Suppose $T$ has a unique centre edge. If $uv$ is the centre edge, then the result follows from the definition of star-root vertex. Now assume that $uv$ is not the centre edge. Then the subtree $T_v$ cannot contain the centre edge, since one endpoint of the centre edge—namely, the vertex $u$- does not lie in $T_v$. Consequently, the subtree $T_u$ contains the centre edge. By Lemma \ref{centre-edge-domination}, we have $\vert V_u \vert > \vert V_v \vert.$
    
    Now assume that $T$ has at least two centre edges. Since $T_v$ can not contain a centre edge, all but at most one centre edge should lie in $T_u$. In particular, the subtree $T_u$ contains at least one centre edge. Therefore, the result follows by Lemma \ref{centre-edge-domination}.
\end{proof}

In the theorem below, we discuss extremal trees when the maximum vertex degree is fixed.
\begin{theorem}\label{extremal-graph-max-vertex-degree}
    Let $T$ be a tree on $n$ vertices with maximum vertex degree $d_{\max}.$       
\begin{enumerate}
    \item If $n$ is even, then
    \[
        b(T) \geq 
        \begin{cases} 
            \frac{2}{n}, & d_{\max} \leq \frac{n}{2}, \\
            \frac{n}{2d_{\max}(n-d_{\max})}, & d_{\max} > \frac{n}{2}.
        \end{cases}
    \]
    \item If $n$ is odd, then
    \[
        b(T) \geq 
        \begin{cases} 
            \frac{2n}{n^2 - 1}, & d_{\max} \leq \lfloor \frac{n}{2} \rfloor + 1, \\
            \frac{n}{2d_{\max}(n-d_{\max})}, & d_{\max} > \lfloor \frac{n}{2} \rfloor + 1.
        \end{cases}
     \]
     \item If $d_{\max}$ divides $n-1$, then $b(T) \leq \frac{n}{2k(n-k)}$ where $n-1=kd_{\max}$ for some $k \in \mathbb{N}$.
     \item If $n-1=k d_{\max}+r$ for some $k,r \in \mathbb{N}$ with $1 \leq r <d_{\max}$, then $b(T) \leq \frac{n}{2(k+1)(n-k-1)}.$
\end{enumerate}
Moreover, all the bounds mentioned above are sharp.
 \end{theorem}
\begin{proof}
Construct a tree $T$ which minimizes $b(T)$ as follows:  Let $v_1$ be a vertex of degree $d_{\max}$, and attach a path on $n-d_{\max}-1$ vertices to one of the neighbors of $v_1$.

\noindent\textbf{Proof of (1): } Let $n=2l$ for some positive integer $l$.
If $d_{\max} \leq n/2$, then there is a sparsest cut induced by exactly ${n/2}$ vertices in $T$. That is, $b(T)=2/n$. So, $T$ minimizes $b(T)$ in this case.


If $d_{\max} > n/2$, then there is a sparsest cut induced by $n-d_{\max}$ number of vertices in $T$. That is, $b(T)=\frac{n}{2d_{\max}(n-d_{\max})}$.  We claim $T$ minimizes $b(T)$ in this case. Suppose that,  $T^\prime=(V^\prime,E^\prime)$ is a tree with $b(T^\prime) <b(T)$. Let $w^\prime \in V^\prime$ with $\deg(w^\prime)=d_{\max}$. Let $u^\prime v^\prime$ be a centre edge in $T^\prime$. From Theorem \ref{extended centre edge result}, we get
 $$
\frac{n}{2}\frac{1}{\vert V^\prime_{u^\prime}\vert \vert V^\prime_{v^\prime}\vert}=b(T^\prime)<b(T)=\frac{n}{2}\frac{1}{d_{\max}(n-d_{\max})}.
 $$
  Therefore, $\vert V^\prime_{u^\prime}\vert \vert V^\prime_{v^\prime}\vert > d_{\max}(n-d_{\max})$. Without loss of generality, let $\vert V^\prime_{u^\prime}\vert \leq \vert V^\prime_{v^\prime}\vert$. Then $n-d_{\max} < \vert V^\prime_{u^\prime}\vert \leq l <d_{\max}$ and $l \leq \vert V^\prime_{v^\prime}\vert <d_{\max}$. Note that there is exactly one edge between $T^\prime_{u^\prime}$ and $T^\prime_{v^\prime}$, so the subtree that containing $w$ should contain at least $d_{\max}-1$ neighbors of $w$. This implies that at least one of the subtrees should be of order at least $d_{\max}$. That is, either $\vert T^\prime_{u^{'}} \vert \geq d_{\max}$ or  $\vert T^\prime_{v^{'}} \vert \geq d_{\max}$, which is  a contradiction. Thus $T$ minimizes $b(T)$.



\noindent\textbf{Proof of (2): } Let $n=2l+1$ for some positive integer $l$.
 Let $d_{\max} \leq \lfloor \frac{n}{2} \rfloor +1$. Then there is a sparsest cut in $T$ induced by $l$ vertices in $T$ with $b(T)=\frac{2n}{n^2-1}$. Consequently, $T$ minimises $b(T)$. If $d_{\max} > \lfloor \frac{n}{2} \rfloor +1$, then there is a sparsest cut induced by $n-d_{\max}$ vertices in $T$ with $b(T)=\frac{n}{2d_{\max}(n-d_{\max})}$. The remainder of the proof follows along the same lines as in part (1).




\noindent\textbf{Proof of (3):} Let $d_{\max}$ divides $n-1$, and let $n-1 = k d_{\max}$. Construct the tree $T$ as follows: Let $\deg(v_1)=d_{\max}$. Attach to each pendant neighbor of $v_1$ a branch of size $k-1$, ensuring that the degree of every vertex in the resulting graph is at most $d_{\max}$. This procedure may produce more than one nonisomorphic tree. Fix any one of the trees obtained in this way, and denote it by $T$. 
 


In $T$, the $k$ vertices of any single branch induces a sparsest cut. Thus, $b(T)=\frac{n}{2k(n-k)}$. We now show that $T$ maximizes $b(T)$. Suppose, to the contrary, that there exists a tree  $T^\prime=(V^\prime,E^\prime)$ such that $b(T^\prime)>b(T)$. Let $w^\prime \in V^\prime$ be a star-root vertex of $T^\prime$, and let $w^\prime v^\prime$ be a centre edge in $T^\prime$. By Theorem \ref{extended centre edge result}
$$
 \frac{n}{2}\frac{1}{\vert V^\prime_{w^\prime}\vert \vert V^\prime_{v^\prime}\vert}=b(T^\prime)>b(T)=\frac{n}{2k(n-k)}.
 $$
Hence $\vert V^\prime_{w^\prime}\vert \vert V^\prime_{v^\prime}\vert < k(n-k)$. By Lemma \ref{centre vertex domination}, we have $\vert V^\prime_{w^\prime}\vert \geq \vert V^\prime_{v^\prime}\vert$. It follows that $$\vert V^\prime_{v^\prime}\vert \leq k-1~\mbox{~ and~} \vert V^\prime_{w^\prime} \vert \geq n-k+1.$$ Therefore, $$\vert V^\prime_{w^\prime} \vert - \vert V^\prime_{v^\prime}\vert \geq n-2k+2.$$
Now we claim that every branch of $w^\prime$ holds at most $k-1$ vertices. Suppose that, there exists a branch of $w^\prime$ that contains at least $k$ vertices. Let $z^\prime$ be the neighbor of $w^\prime$ in this branch. By Lemma \ref{centre vertex domination}, we have $\vert V^\prime_{w^\prime}\vert \geq \vert V^\prime_{z^\prime} \vert$. Also  $\vert V^\prime_{w^\prime}\vert \leq n-k$ and $\vert V^\prime_{z^\prime}\vert \geq k$. Therefore $$\vert V^\prime_{w^\prime}\vert-\vert V^\prime_{z^\prime}\vert \leq n-2k <n-2k+2 \leq \vert V^\prime_{w^\prime} \vert - \vert V^\prime_{v^\prime}\vert,$$
which contradicts the assumption that $w^\prime v^\prime$ is a centre edge of $T^\prime$.
Since $deg(w^\prime) \leq d_{\max}$, the number of vertices in $T^\prime$ is at most $d_{\max}(k-1)+1 <d_{\max}k+1=n$, which is a contradiction. Therefore $T$ maximizes $b(T)$.

\noindent\textbf{Proof of (4):} Let $n-1=d_{\max}k+r$, where $k$ and $r$ are non-negative integers with $1 \leq r <d_{\max}$.
In this case, we use a construction for $T$ similar to that in the previous case. In addition, the remaining $r$ vertices are distributed among distinct branches of $v_1$, one vertex per branch. As a result, exactly $r$ branches containing $k+1$ vertices each, while the remaining branches contain exactly $k$ vertices each.


 The set of vertices in any branch of $T$ containing $k+1$ vertices induces a sparsest cut in $T$.  Hence, $b(T)=\frac{n}{2(k+1)(n-k-1)}$. Suppose, for contradiction, that there exists a tree $T^\prime=(V^\prime,E^\prime)$ such that $b(T^\prime) >\frac{n}{2(k+1)(n-k-1)}$. Let $w^\prime \in V^\prime$ be a star-root vertex of $T^\prime.$ Arguing as in the previous case, we conclude that every branch of $w^\prime$ has at most $k$ vertices. Hence, $T^\prime$ has at most $kd_{\max}+1<kd_{\max}+1+r=n$ vertices, which is a contradiction. Thus, $T$ maximizes $b(T)$.
\end{proof}

The next theorem concerns extremizers for trees of fixed number of pendant vertices.

\begin{theorem}\label{pendant-vertex-extremizer}
   Let $T$ be a tree on $n$ vertices with $p$ pendant vertices. Then, we have the following:  
    \begin{enumerate}
  \item $b(T) \geq \frac{n}{2\lceil \frac{n}{2}\rceil \lfloor \frac{n}{2}\rfloor}$.
      \item If $p$ divides $(n-1)$,  write $n-1=kp$, then $b(T) \leq \frac{n}{2k(n-k)}$.
       \item If $n-1=kp+r$ for some $k,r \in \mathbb{N}$ with $1 \leq r < p$, then $b(T) \leq \frac{n}{2(k+1)(n-k-1)}$.
   
   \end{enumerate}Furthermore, each of the above bounds is attained, and hence all the bounds are sharp.
\end{theorem}
\begin{proof}

\textbf{Proof of (1):} Construct a tree $T$ as follows:  Start with a path on  $n-p+2$ vertices, and  let $uv$ be a centre edge of the path. If $n-p+2$ is even,  then $uv$ is the only centre edge. In this case,  attach the remaining $p-2$ pendant vertices with the vertices $u$ and $v$, alternately, starting from either $u$ or $v$. If $n-p+2$ is odd, then, without loss of generality, assume $\vert V_u \vert = \vert V_v\vert -1.$ In this case attach the pendant vertices alternately to $u$ and $v$, starting from $u$, until all the remaining vertices are exhausted. In both cases, by Lemma \ref{centre edge unchanged}, the edge $uv$  remains  a centre edge of $T$. Moreover deleting the edge $uv$ gives two subtrees whose orders differ by at most one. Consequently, $b(T) = \frac{n}{2\lceil \frac{n}{2}\rceil \lfloor \frac{n}{2}\rfloor}$. Thus $T$ minimizes $b(T)$.



\noindent\textbf{Proof of (2):} Assume that $p$ divides $n-1$, and write $n-1 = kp$. Consider the star graph $S_{p+1}$ on $p+1$ vertices, attach to each pendant vertex of $S_{p+1}$ a path on $k-1$ vertices. Denote the resulting tree by $T$. The vertices in any of the single branch of $T$ induces a sparsest cut, and hence $b(T)=\frac{n}{2k(n-k)}$. We claim that $T$ maximizes $b(T)$ in this case. Suppose that, there exists a tree $T^\prime=(V^\prime,E^\prime)$ such that $b(T^\prime)>b(T)$. Let $w^\prime \in V^\prime$ be a star-root vertex of $T^\prime.$ Since $T^\prime$ has $p$ pendant vertices, $w^\prime$ has at most $p$ branches; in particular,  $\deg(w^\prime)\leq p$. The remainder of the argument is analogous to that of part (3) of Theorem \ref{extremal-graph-max-vertex-degree}. 
\noindent\textbf{Proof of (3):} Let $n-1=kp+r$ where $k\in \mathbb{N}$ and $r$ is a non-negative integer with $1 \leq r < p$. Start with the star $S_{p+1}$, and attach a path on $k$ vertices to each of the $r$ branches and attach a path on $k-1$ vertices to the remaining of $p-r$ branches.  In the resulting tree, denoted by $T$, wthe vertices of any branch containing $k+1$ vertices induce a sparsest cut. Therefore, $b(T)=\frac{n}{2(k+1)(n-k-1)}$. The argument showing that $T$ maximizes $b(T)$ is analogous to the proof of part (4) of Theorem \ref{extremal-graph-max-vertex-degree}.




\end{proof}

\section{Connection with Laplacian matrices}
In this short section, we establish a connection between the quantity $b(G)$ and the Laplacian matrix of the graph $G$. Note that, if $S$ induces a sparsest cut, the associated $l_1$-\textit{Fiedler vector} \cite{enide-geir-2024} is given by $$ x_v =
\left\{
	\begin{array}{ll}
		\frac{1}{2|S|}  & \mbox{if } v \in S, \\
	 -\frac{1}{2|S^c|} & \mbox{if } v \in S^c.
	\end{array}
\right.
$$
\begin{theorem}\label{lap-conn}
Let $G$ be a graph, and let $S \subseteq V(G)$ induce a sparsest cut in $G$. Let $x$ be the associated $l_1$-Fiedler vector. Then we have the following:
\begin{itemize}
    \item[(i)] $\sum\limits_{u \in S}(L(G)x)_u=b(G),$
    \item[(ii)] $\sum\limits_{u \in S^c}(L(G)x)_u=-b(G)$,
\end{itemize}    
where $L(G)$ is the Laplacian matrix of $G$.
\end{theorem}
\begin{proof}

Let $l_{ij}$ denotes the $ij$-th entry of $L(G)$. For a vertex $u\in S$, let $\vert N_S(u) \vert$ and $\vert N_{S^c}(u) \vert$ represent the number of neighbors of $u$ in $S$ and $S^c$, respectively. Therefore, 
\begin{align*}
    (L(G)x)_u &=\sum_{w \in V(G)} l_{uw}x_{w}\\
    &=l_{uu}x_u+\sum_{w \neq u} l_{uw}x_{w}\\
    &=\frac{\deg(u)}{2\vert S \vert }-\Bigl(\sum_{\substack{w \sim u \\ w\in S}} x_w +\sum_{\substack{w \sim u \\ w\in S^c}} x_w\Bigl)\\ 
    & =\frac{\deg(u)}{2\vert S \vert}-\Bigl(\frac{\vert N_S(u) \vert}{2\vert S \vert }-\frac{\vert N_{S^c}(u) \vert}{2\vert S^c \vert}\Bigl)\\
    &=\frac{\vert N_{S^c}(u) \vert}{2\vert S \vert }+\frac{\vert N_{S^c}(u) \vert}{2\vert S^c \vert}\\
    &= \vert N_{S^c}(u) \vert \frac{n}{2\vert S\vert \vert S^c \vert}.
\end{align*}
By the symmetry, we  get 
$
b(G)= \frac{n}{2} \frac{\vert \partial S^c\vert}{\vert S\vert \vert S^c \vert}
$.
So we have 
$$
(L(G)x)_u=b(G) \frac{\vert N_{S^c}(u) \vert}{\vert \partial S^c \vert}.
$$
Similarly, if $u \in S^c$, we have
$$
(L(G)x)_u=-b(G) \frac{\vert N_{S}(u)\vert}{\vert \partial S \vert}.
$$
By summing the above expression over all vertices of $S$ and  $S^c$, respectively, we get the result.
\end{proof}

\section{Relationship with Other Graph Parameters}

In this section, we establish connections between the quantity $b(G)$ and two fundamental graph parameters: the edge connectivity and the isoperimetric number of the graph.
\subsection{Edge-connectivity}
     The edge-connectivity of a graph $G$ is the minimum number of edges whose removal disconnects $G$. In the following result, we establish an upper bound for $b(G)$ in terms of edge connectivity.
\begin{theorem}\label{edge-connc-b(g)} Let $G$ be a connected graph on $n$ vertices with edge-connectivity $k$. Then
$$
b(G) \leq \frac{nk}{2(n-1)}.
$$
\end{theorem}

\begin{proof}
    If $G$ is the complete graph $K_n$, then the result is immediate. Let $G$ be a graph on $n$ vertices, other than $K_n$. Let $S\subseteq V(G)$ be such that $b(G)=\frac{n}{2} \rho(S)$. Add an edge between two non-adjacent vertices of $G$, and let $G_1$ be the resultant graph. Then, 
    \begin{align*}
        b(G_1) &\leq \frac{n}{2} \frac{\vert \partial S_{G_1} \vert}{\vert S\vert(n-\vert S\vert)}\\
        &\leq \frac{n}{2} \frac{\vert \partial S_G \vert+1}{|S|(n-\vert S \vert)}\\
        &=b(G)+\frac{n}{2\vert S \vert (n-\vert S\vert )}\\
        &\leq b(G) + \frac{n}{2(n-1)}.
    \end{align*}
    So if we add $k$ edges to the graph $G$, then for the resulting graph $G_k$, we have $b(G_k) \leq b(G) + \frac{nk}{2(n-1)}$.
    
    Let $G$ be a graph with edge-connectivity $k$. Then there exists a set of $k$ edges such that the graph obtained from $G$ by removing these edges, say $G_{-k}$, is disconnected. As $G_{-k}$ is disconnected,  $b(G_{-k}) = 0$.  Thus,

    $$
b(G) \leq b(G_{-k}) +\frac{nk}{2(n-1)}=\frac{nk}{2(n-1)}.
    $$
    
\end{proof}
Note that the above bound is tight. Among trees, equality holds only for star graphs (see Theorem \ref{extremal trees}). The bound also holds with equality for the complete  graph $K_n$. In the next theorem, we characterize the structure of all the graphs attaining this bound.
\begin{theorem}\label{edge-connc-b(g)-equi-chara}
   Let $G$ be a graph with edge-connectivity $k$, and define $G_0=G$. Then there exists a set of $k$ edges such that by removing these  $k$ edges one after another, the resultant graph $G_{-k}$ is disconnected. Let $G_{-s}$ denote the graph obtained by deleting $s$ edges from the graph $G$ in this process, for $s \in \{1,2,\dots,k\}$. Then  \begin{equation*}
b(G) = \frac{nk}{2(n-1)}
\end{equation*} if and only if 
\begin{enumerate}
    \item[(a)] there exists an isolated vertex $v$ in $G_{-k}$ such that adding all the $k$ edges to $v$ gives the graph $G$, and 
    \item[(b)] in every step $\{v\}$ induces a sparsest cut in each of the graph $G_{-k+i}$ for $i \in \{1,2,\dots,k\}.$
\end{enumerate}
\end{theorem}

\begin{proof}
Let $v$ be such an isolated vertex in $G_{-k}$. We add $k$ edges and then obtain $G$. In the last step, also $\{v\}$ induced a sparsest cut.
Therefore $b(G)=\frac{nk}{2(n-1)}$. Thus, we establish the sufficiency part.\\
To prove the necessary condition, in Theorem \ref{edge-connc-b(g)} we need the equalities to be held at every step as we add the edges to $G_{-k}$ in   reverse order. Let $S$ induce a sparsest cut in $G_{-k}$. Consider the first step, if $b(G_{-k+1})=\frac{n}{2(n-1)}$, by Theorem \ref{edge-connc-b(g)}, we must have
\begin{equation}\label{first}
    b(G_{-k+1})= \frac{n}{2} \frac{\vert\partial S_{G_{-k+1}}\vert}{\vert S\vert(n-\vert S \vert)},
\end{equation}
\begin{equation}\label{middle}
    \frac{n}{2} \frac{\vert\partial S_{G_{-k+1}}\vert}{\vert S\vert(n-\vert S \vert)} = \frac{n}{2} \frac{\vert \partial S_{G_{-k} }\vert+1}{\vert S \vert (n-\vert S \vert)},
\end{equation}
and 
\begin{equation}\label{last}
   \frac{n}{2\vert S \vert (n-\vert S\vert )} =  \frac{n}{2(n-1)}.
\end{equation}
 From (\ref{last}), we have $S$ must be a set with single vertex in $G_{-k}$, and it is an isolated vertex in $G_{-k}$. As equality holds in  (\ref{middle}), we add an edge to $v$ to get $G_{-k+1}$. From Equation (\ref{first}), the vertex $v$ induces a sparsest cut in $G_{-k+1}$.
 
Now it is easy to see that the very next edge we should add to $\{v\}$. If not, let the edge not contain $v$ as an endpoint, then after removing $k-2$ edges, it leads to $G_{-k+2}$, and removing the very first edge that we have added to $v$, makes the graph $G$ disconnected. So, $G$ has edge connectivity strictly less than $k$. Thus the result follows.
\end{proof}






A natural question that arises at this stage is how the parameter $b(G)$ changes when new vertices are added to the graph. Next, we show that adding pendant vertices strictly decreases the quantity $b(G)$.

\begin{theorem}\label{vertex-add}
 Let $G$ be a connected graph on $n$ vertices. Let $G^k$ be a graph obtained by adding $k$ pendant vertex to the graph $G$. Then we have 
 $$
b(G^k) \leq b(G) \prod_{i=0}^{k-1} \Bigl(1-\frac{1}{(n+i)^2}\Bigl).
 $$

\end{theorem}

\begin{proof}
First, we consider the effect of adding a pendant vertex $v$ to $G$.  Let the resultant graph be $G^1$. Let $S$ induces a sparsest cut of $G$ with $\vert S\vert \leq \vert S^c \vert$.\\
\textbf{Case 1:} Let the neighbor of the new vertex $v$ lie in $S$. Consider the set $S^{\prime}=S \cup \{v\}$. Then
\begin{align*}
b(G^1) &\leq \frac{n+1}{2} \frac{\vert \partial S^{\prime}_{G^1}\vert}{\vert S^{\prime}\vert(n+1-\vert S^{\prime}\vert)}\\
&= \frac{n+1}{2} \frac{\vert \partial S_{G}\vert}{(\vert S\vert+1)(n-\vert S\vert)}  .  
\end{align*}
\textbf{Case 2:} Let the neighbor of the pendent vertex lie in $S^c$. Then
\begin{align*}
   b(G^1) &\leq \frac{n+1}{2} \frac{\vert \partial S_{G^1}\vert}{(\vert S \vert )(n+1-\vert S \vert)}\\
    &=\frac{n+1}{2} \frac{\vert \partial S_{G}\vert}{\vert S \vert (n+1-\vert S \vert)} . 
\end{align*}
From the last two cases, we have

$$
b(G^1) \leq \frac{n+1}{2} \max \{ \frac{\vert \partial S_{G} \vert}{\vert S\vert (n+1-\vert S \vert )}, \frac{\vert \partial S_{G} \vert }{(\vert S\vert+1)(n-\vert S \vert )}\}.
$$
Note that if $\vert S \vert =\vert S^c \vert $ then we have $n=2\vert S \vert $, which means $\vert S \vert (n+1-\vert S \vert )=(\vert S \vert +1)(n-\vert S \vert )$.\\
If $\vert S \vert < \vert S^c \vert$ then we have $\vert S \vert(n+1-\vert S \vert)< (\vert S \vert +1)(n-\vert S \vert)$. Together we get

\begin{align*}
b(G^1) & \leq \frac{n+1}{2} \frac{\vert \partial S_{G}\vert}{\vert S \vert (n+1-\vert S \vert)}\\
& =\frac{n+1}{2} \frac{\vert \partial S_{G} \vert }{\vert S \vert (n-\vert S \vert )} \frac{(n-\vert S \vert )}{(n+1-\vert S \vert )}\\
& \leq \frac{n+1}{2} \frac{2b(G)}{n} \frac{n-1}{n}\\
&=b(G) (1-\frac{1}{n^2}).
\end{align*}
The desired result follows by repeating the process $k$ times.
\end{proof}
 
\begin{remark}
Note that the upper bound can be attained; one such example is $S_n$ where the pendant vertex is attached to the center of the star, resulting in creating $S_{n+k}$. It is easy to check that $S_n$ is the only tree for which the inequality turns into equality(See Remark \ref{1}).
\end{remark}

\subsection{Isoperimetric number}\label{iso-connection}
 Next, we establish a bound that connects the isoperimetric number of a graph $G$ with the quantity $b(G)$.
\begin{theorem}\label{b(g)-iso}
    Let $G$ be a graph on $n$ vertices. Then $b(G) \leq \iso(G)$.
\end{theorem} 
\begin{proof}
    Let $S$ be an isoperimetric set of $G$. Then,
\begin{align*}
   b(G) &\leq \frac{n}{2} \frac{\vert \partial S\vert}{\vert S\vert(n-\vert S\vert)} \\ 
   &= \frac{n}{2} \frac{\iso(G)}{(n-\vert S\vert)} \\
   &\leq \iso(G).
\end{align*}
   
\end{proof}
In \cite{mohar-isoperi-jctb}, Mohar proved the following.
\begin{theorem}[{\cite[Theorem 4.2]{mohar-isoperi-jctb}}] \label{iso-a(G)}
If $G$ is a graph with at least $4$ vertices, then 
$$
\iso(G) \leq \sqrt{a(G)(2 d_{\max}(G) - a(G))},
$$
\end{theorem}

\begin{cor}
Let $G$ be a graph on $n$ vertices. Then, by Theorem \ref{b(g)-iso} and Theorem \ref{iso-a(G)}, we get \begin{equation}\label{b(g)-a(g)-compare}
 b(G) \leq \sqrt{a(G)(2 d_{\max}(G) - a(G))}.    
\end{equation} 
\end{cor}

\begin{remark}
In \cite{enide-geir-2024}, the authors show that \begin{equation}\label{b(g)-a(g)-m-compare}
    b(G) \leq \sqrt{ma(G)}.
    \end{equation}
    Let $G$ be a connected r-regular graph with at least 4 vertices, then 
    $$
a(G)>\frac{r(4-n)}{2}=2r-\frac{nr}{2}=2r-m
    $$
    which shows 
    $$
a(G)(2r-a(G)) < m(a(G))
    $$
    So, the bound in (\ref{b(g)-a(g)-compare}) works better than the bound in (\ref{b(g)-a(g)-m-compare}) for many graphs.
\end{remark}
 The cube graph $Q_n$ consists of vertices represented by binary strings of length $n$,where two vertices are connected by an edge if their strings differ in exactly one bit. Thus $Q_n$ has $2^n$ vertices and it is  $n$-regular. we also have $a(Q_n)=2$\cite{alg-con}. 
As an application of Theorem \ref{b(g)-iso}, we next calculate $b(G)$ explicitly for a couple of graphs. 
\begin{cor}
\begin{enumerate}
    \item[(a)]     Let $Q_n$ be the $n$-dimensional hypercube graph. Then $b(Q_n)=1$.
\item[(b)]    Let $P$ be the Petersen graph. Then $b(P)=1$.

\end{enumerate}
\end{cor}
\begin{proof}
   \textbf{$(a)$} It is known that  $a(Q_n)=2$  \cite{alg-con}. Therefore, by Theorem \ref{b(g)-min-deg-cutsize}, $b(Q_n) \geq 1$. By Theorem \ref{b(g)-iso}, we have $$b(Q_n) \leq \iso (Q_n).$$ Thus, $b(Q_n) \leq 1$ as $\iso(Q_n)=1$ \cite[Example $2.9$]{nica-book-2018}. Thus $b(Q_n)=1$.

   \textbf{$(b)$}   From \cite{old-new-a(G)}, we have $a(P)=2$. Therefore, by Theorem \ref{b(g)-min-deg-cutsize}, $b(P) \geq 1$. By Theorem \ref{b(g)-iso}, we have $b(P) \leq \iso(P)$. Thus $b(P) \leq 1$ as $\iso(P)=1$\cite{nica-book-2018}. Thus $b(P)=1$.
\end{proof}

In both corollaries above, we have $b(G)=\iso(G)$. The following results provide a method for constructing graphs where $b(G)$ coincides with $\iso(G)$.

\begin{lemma}
    
\label{iso(G)=b(g) lemma}
    Let $G$ be a graph on $n$ vertices and $S$ induces a sparsest cut with $ \vert S \vert=\left\lfloor \frac{n}{2} \right\rfloor$. Then $S$ is an isoperimetric set.
\end{lemma}
\begin{proof}
    Let $T \subseteq V(G)$ with $\vert T \vert \leq \left\lfloor \frac{n}{2} \right\rfloor$. Then we have  $$\frac{\vert{\partial S}\vert}{\vert S\vert (n-\vert S \vert)} \leq \frac{\vert {\partial T}\vert}{\vert T\vert(n- \vert T\vert)}.$$ Multiplying both side by $n- \vert S \vert$, we get
    $$
\frac{\vert{\partial S}\vert}{ \vert S \vert } \leq \frac{\vert \partial T \vert}{\vert T \vert} \frac{(n- \vert S \vert)}{(n- \vert T \vert)} \leq \frac{\vert \partial T \vert}{\vert T \vert}.
    $$
     Hence, $S$ must be an isoperimetric set.
\end{proof}
\begin{theorem}\label{iso(G)=b(G)}
    Let $G$ be a graph on $n$ vertices, where $n$ is an even number. If $S$ induces a sparsest cut with $ \vert S \vert= \frac{n}{2} $, then $\iso(G)=b(G)$. 
\end{theorem}
\begin{proof}
    By Lemma \ref{iso(G)=b(g) lemma}, we have that $S$ is an isoperimetric set. Hence
    $$
\iso(G)=\frac{\vert{\partial S}\vert}{|S|}=\frac{n}{2} \times \frac{\vert\partial S \vert}{\frac{n}{2}\vert S \vert}=\frac{n}{2} \frac{\vert\partial S \vert}{ (n - \vert S \vert)\vert S \vert}=b(G).
    $$
\end{proof}
\begin{remark}
    From Theorem \ref{iso(G)=b(G)}, it is easy to see that if we take two graphs with an equal number of vertices and then join them by exactly one edge, the resultant graph $G$ will satisfy $\iso(G)=b(G)$. 
\end{remark}
\begin{theorem}
Let $G$ be a graph such that $S$ is an isoperimetric singleton set. Then $S$ induces sparsest cut in $G$.
\end{theorem}
\begin{proof}
Let $S=\{v\}$. Then for any subset $T$ of $V(G)$ with $|T| \leq \frac{n}{2}$, we have
$$
\vert \partial S \vert=\frac{\vert\partial S\vert}{\vert S \vert}\leq \frac{\vert \partial T \vert}{\vert T \vert}.
$$
But
$$
\rho(T)=\frac{\vert {\partial T}\vert}{\vert T\vert(n- \vert T\vert)} \geq \frac{\vert{\partial S}\vert}{\vert S \vert (n-\vert T \vert)} \geq \frac{\vert{\partial S}\vert}{|S|(n-\vert S \vert)}=\rho(S).
$$
Hence, $S$ induces sparsest cut.
\end{proof}






We now derive a lower bound on $b(G)$ in terms of the isoperimetric number for regular graphs.

Let $S \subseteq V(G)$, volume of $S$, denoted by  $\vol S$ is defined as follows: 
$$
\vol  S= \sum_{u \in S }\deg(u).
$$
The \textit{Cheegar's constant} of $G$ is defined as follows \cite{chung-book-1997}: $$h(G) =  \min_S \frac{|\partial S|}{\min(\vol S, \vol S^c)}$$

\begin{theorem}[{\cite[Corollary 2.9]{chung-book-1997}}] \label{Cheeger's const bound}
 For an r-regular graph $G$, we have
$$
r\cdot h(G)\geq \inf\limits_x \frac{\sum\limits_{u \sim v} \vert x_u-x_v \vert}{\sum\limits_{u \in V(G)}\vert x_u\vert } \geq \frac{1}{2}r\cdot h(G),
$$ 
where $ x \in \mathbb{R}^n$ satisfying
$$
\sum_{u\in V(G)}x_u=0.
$$
\end{theorem}

\begin{theorem}
    Let $G$ be a regular graph. Then $$ b(G) \geq \frac {\iso(G)}{2}.$$
\end{theorem}

\begin{proof}
    
Let $G$ be an $r$-regular graph. Let $x$ be an $l_1$-Fiedler vector. Then,  
$$
\sum_{v \in V(G)} x_v=0, ~~~~~~\sum _{v \in V(G)} \vert x_v \vert =1, ~~~\mbox{and}~~~ b(G)=\sum_{uv \in E(G)} \vert x_u-x_v \vert.
$$

Now, 
\begin{align*}
  h(G) &=\min_S \frac{\vert \partial S \vert}{\min(\vol\,S, \vol\,S^c)} \\
  &=\frac{1}{r}\min_S \frac{\vert\partial S\vert }{\min(\vert S \vert, \vert S^c \vert)}\\
    &=\frac{1}{r}\min_S \Bigl\{\frac{\vert\partial S\vert}{\vert S\vert} : \vert S \vert \leq \frac{n}{2}  \Bigl\}\\
  &=\frac{\iso(G)}{r}.
\end{align*}\\
Then, by Theorem \ref{Cheeger's const bound},
$$
b(G)=\sum_{uv \in E(G)} \vert x_u-x_v \vert \geq \frac{1}{2}r\cdot h(G) \sum_{v \in V(G)} \vert x_v\vert  =\frac{r\cdot h(G)}{2}=\frac{1}{2}\iso(G).
$$
\end{proof}

 \section*{Acknowledgments}
  M. Rajesh Kannan acknowledges financial support from the ANRF-CRG India and  SRC, IIT Hyderabad. Rahul Roy thanks the University Grants Commission (UGC), India, for financial support through a Junior Research Fellowship.
	\bibliographystyle{amsplain}
	\bibliography{b(G)_ref}
\end{document}